\documentclass[envcountsame,envcountsect,notbib]{svmult}
\usepackage{makeidx}         % allows index generation
\usepackage{graphicx}        % standard LaTeX graphics tool
\usepackage{multicol}        % used for the two-column index
\usepackage[bottom]{footmisc}
\usepackage{url}
\makeindex
\usepackage{amscd,xy,amssymb,amsmath,nicefrac}

\def\Kweb#1{   %for bibliography
http:\linebreak[2]//www.\linebreak[2]math.\linebreak[2]uiuc.%
\linebreak[2]edu/\linebreak[2]{K-theory/\linebreak[2]#1}}
\def\cB{\mathcal B}

\def\cF{\mathcal F}

\def\cI{\mathcal I}

\def\cM{\mathrm{M}} % was \mathcal M}

\def\cT{\mathcal T}
\newcommand{\fin}{\mathfrak{Fin}}
\newcommand{\comment}[1]{}

\newcommand{\GL}{\mathrm{GL}}
\renewcommand{\top}{\rm top}

\newcommand{\Q}{\mathbb{Q}}

\newcommand{\Z}{\mathbb{Z}}
\newcommand{\frakg}{\mathfrak{g}}
\newcommand{\g}{\frakg}
\newcommand{\frakt}{\mathfrak{t}}
\newcommand{\fS}{\mathfrak{S}}

\newcommand{\rE}{\mathrm{E}}
\newcommand{\rB}{\mathrm{B}}

\def\rht{\mathrm{rht}}

\def\norm{\mathrm{norm}}

\def\coker{\operatorname{coker}}

\def\tr{\text{tr}} % was \widetilde{Tr}
\def\map#1{\ {\buildrel #1 \over \lra}\ }
\def\smap#1{\ {\buildrel #1 \over \to}\ }
\def\lmap#1{\ {\buildrel #1 \over \llra}\ }
\def\lra{\longrightarrow}
\def\llra{\relbar\joinrel\longrightarrow}
\def\sg{\mathrm{sg}}

\newcommand{\eps}{\varepsilon}

\newcommand{\Ha}{\mathsf{H}}
\newcommand{\haha}{\hat{\Ha}}
\newcommand{\hotimes}{\hat{\otimes}}
\newcommand{\iso}{\overset{\cong}{\to}}
\input xypic
\xyoption{all} %\CompileMatrices
\numberwithin{equation}{subsection}

\spnewtheorem{thm}[equation]{Theorem}{\bf}{\it}
\spnewtheorem{prop}[equation]{Propostion}{\bf}{\it}
\spnewtheorem{lem}[equation]{Lemma}{\bf}{\it}
\spnewtheorem{cor}[equation]{Corollary}{\bf}{\it}

\spnewtheorem{rem}[equation]{Remark}{\it}{\rm}
\spnewtheorem{ex}[equation]{Remark}{\it}{\rm}
\spnewtheorem{defn}[equation]{Definition}{\bf}{\it}

\begin{document}
\bibliographystyle{plain}

\title*{Relative Chern characters for nilpotent ideals}
\titlerunning{Relative Chern character}

\author{G. Corti\~nas \inst{1} \and C. Weibel \inst{2}}
\institute{Dep. Matem\'atica, Ciudad Universitaria Pab 1, C1428EGA Buenos Aires, Argentina. \texttt{gcorti@dm.uba.ar}
\and Dept.\ of Mathematics, Rutgers University, New Brunswick,
NJ 08901, USA. \texttt{weibel@math.rutgers.edu}}
\thanks{Corti\~nas is supported by FCEyN-UBA and CONICET,
and partially supported by grants
ANPCyT PICT 2006-0836, UBACyT-X057, and MEC MTM00958.}

%\institute{Name and Address of your Institute
%\texttt{name@email.address}
%\and Name and Address of your Institute \texttt{name@email.address}}

\thanks{\hskip-8pt Weibel's research was %partially
supported by NSA grant MSPF-04G-184 and NSF grant DMS-0801060}

%\date{\today}
\maketitle
\numberwithin{equation}{section}
\section{Introduction}

When $A$ is a unital ring, the {\it absolute Chern character}
is a group homomorphism
$%\begin{equation}\label{intro:cha}
ch_*: K_*(A)\to HN_*(A)
$, %\end{equation}
going from algebraic $K$-theory to negative cyclic homology (see
\cite[11.4]{lod}). There is also a relative version, defined for any ideal
$I$ of $A$:
\begin{equation}\label{intro:ch}
ch_*: K_*(A,I)\to HN_*(A,I).
\end{equation}
Now suppose that $A$ is a $\Q$-algebra and that $I$ is nilpotent.
In this case, Goodwillie proved in \cite{goo} that
%if $I$ is nilpotent and $\Q\subset A$, then
\eqref{intro:ch} is an isomorphism. His proof uses another character
\begin{equation}\label{intro:ch'}
ch'_*:K_*(A,I)\to HN_*(A,I)
\end{equation}
which is defined only when $I$ is nilpotent and $\Q\subset A$.
Goodwillie showed that \eqref{intro:ch'} is an isomorphism whenever
it is defined, and that
it coincides with \eqref{intro:ch} when $I^2=0$. Using this and a
5-Lemma argument, he deduced that \eqref{intro:ch} is an isomorphism for
any nilpotent $I$. The question of whether $ch_*$ and $ch'_*$ agree for
general nilpotent ideals $I$ was left open in \cite{goo}, and announced
without proof in \cite[11.4.11]{lod}.
This paper answers the question, proving in Theorem \ref{thm:main} that
\begin{equation}\label{intro:agree}
ch_*=ch_*'\text{ for all nilpotent ideals $I$.}
\end{equation}

Here are some applications of \eqref{intro:agree}.
Let $I$ a nilpotent ideal in a commutative $\Q$-algebra $A$.
Cathelineau proved in \cite{cath} that \eqref{intro:ch'}
preserves the %so-called Hodge (or $\lambda$-)
direct sum decomposition coming from the eigenspaces of $\lambda$-operations
and/or Adams operations. By \eqref{intro:agree}, the relative Chern
character \eqref{intro:ch} preserves the direct sum decomposition for
nilpotent ideals. Next, Cathelineau's result and \eqref{intro:agree} are
used in \cite{chwinf} to prove that the absolute Chern character $ch_*$
also preserves the direct sum decomposition. In addition, our result
 \eqref{intro:agree} can be used to strengthen
Ginot's results in \cite{ginot}.

This paper is laid out according to the following plan.
In section \ref{sec1} we show that the bar complex $\rB(\Ha)$
of a cocommutative Hopf algebra $\Ha$ has a natural cyclic module structure.
The case $\Ha=k[G]$ is well known, and the case of enveloping
algebras is implicit in \cite{kas}. In section \ref{sec2}, we relate
this construction to the usual cyclic module of the algebra underlying $\Ha$.
In section \ref{sec3} we consider a Lie algebra $\g$ and factor the
Loday-Quillen map $\wedge^n\g\to C^\lambda_{n-1}(U\g)$ through our construction.
In section \ref{sec4} we consider a nilpotent Lie algebra $\g$ and
its associated nilpotent group $G$ and relate the constructions for
$U\g$ and $\Q[G]$ using the Mal'cev theory of \cite[App. A]{Q-RHT}.
In section \ref{sec5}, we review the definitions of $ch_*$ and $ch'_*$
and prove our main theorem, that \eqref{intro:agree} holds.

\numberwithin{equation}{subsection}

\subsection*{Notation}
If $M=(M_*,b,B)$ is a mixed complex (\cite[2.5.13]{lod}), we will write
$HH(M)$ for the chain complex $(M_*,b)$, and $HN(M)$ for the total complex
of Connes' left half-plane $(b,B)$-complex (written as $\cB M^-$ in
\cite[5.1.7]{lod}). By definition, the homology of $HH(M)$
is the Hochschild homology $HH_*(M)$ of $M$, the homology  of $HN(M)$
is the negative cyclic homology $HN_*(M)$ of $M$, and the projection
$\pi:HN(M)\to HH(M)$ induces the canonical map $HN_*(M)\to HH_*(M)$.

We refer the reader to \cite[2.5.1]{lod} for the notion of a cyclic module.
There is a canonical cyclic $k$-module $C(A)$ associated to
any algebra $A$, with $C_n(A)=A^{\otimes n+1}$, whose underlying
simplicial module $(C,b)$ is the Hochschild complex (see \cite[2.5.4]{lod}).
We will write $HH(A)$ and $HN(A)$ in place of the more awkward expressions
$HH(C(A))$ and $HN(C(A))$.

We write $C_{\ge n}$ for the good truncation $\tau_{\ge n}C$ of a chain
complex $C$ \cite[1.2.7]{chubu}. If $n\ge 0$, then $C_{\ge n}$ can and will
be regarded as a simplicial module via the Dold-Kan correspondence
\cite[8.4]{chubu}.
%If the boundary map $C_1\to C_0$ is zero, then $C_{\ge1}$
%coincides with the brutal truncation of $C$; for example this is
%the case when $\Ha$ is a Hopf algebra and $C=\rB(\Ha)$.

\medskip

\section{Cyclic homology of cocommutative Hopf algebras}\label{sec1}

If $A$ is the group algebra of a group $G$, then the bar
resolution $\rE(A)=k[\rE G]$ admits a cyclic $G$-module structure and the
bar complex $\rB(A)=k[\rB G]$ also admits a cyclic $k$-module structure
\cite{lod}.  In this section, we show that the cyclic modules
$\rE(A)$ and $\rB(A)$ can be defined for any cocommutative Hopf algebra $A$.

%In the next section, this construction will be compared to the
%canonical cyclic $k$-module $C(A)$ associated to any algebra
%$A$, whose underlying simplicial module $(C,b)$ is the Hochschild
%complex, and $C_n(A)=A^{\otimes n+1}$ (see \cite[2.5.4]{lod}).

\subsection{Bar resolution and bar complex of an augmented algebra}
\label{subsec:bars}
Let $k$ be a commutative ring, and $A$ an augmented unital
$k$-algebra, with augmentation $\epsilon:A\to k$. We write $\rE(A)$
for the {\it bar resolution} of $k$ as a left $A$-module
(\cite[8.6.12]{chubu}); this is the simplicial $A$-module
$\rE_n(A)=A^{\otimes n+1}$, whose face and degeneracy operators
are given by
\begin{gather}
\mu_i:\rE_n(A)\to \rE_{n-1}(A)\qquad( i=0,\dots,n)\nonumber\\
\mu_i(a_0\otimes\dots\otimes a_n)=a_0\otimes\dots\otimes
a_ia_{i+1}\otimes\dots\otimes a_n\qquad (i<n)\nonumber\\
\mu_n(a_0\otimes\dots\otimes a_n)=\eps(a_n)a_0\otimes\dots\otimes a_{n-1}
\label{map:simpE}\\
s_j:\rE_n(A)\to \rE_{n+1}(A)\qquad (j=0,\dots, n)\nonumber\\
s_j(a_0\otimes\dots\otimes a_n)=a_0\otimes\dots \otimes
a_{j}\otimes 1\otimes a_{j+1}\otimes\dots\otimes a_n\nonumber
\end{gather}
We write $\partial'$ for the usual boundary map
$\sum_{i=0}^n(-1)^i\mu_i:\rE_n(A)\to \rE_{n-1}(A)$.
The augmentation induces a quasi-isomorphism $\epsilon:\rE(A)\to k$.
The unit $\eta:k\to A$ is a $k$-linear homotopy inverse of $\epsilon$;
we have $\epsilon\eta=1$ and the extra degeneracy
$s:\rE(A)_n\to\rE(A)_{n+1}$,
\begin{equation}\label{map:s}
s(x)=1\otimes x
\end{equation}
satisfies $1-\eta\epsilon=[\partial',s]$. The {\it bar complex} of $A$ is
$\rB(A)=k\otimes_A\rE(A)$; $\rB_n(A)=A^{\otimes n}$.
We write $\partial=1\otimes\partial':\rB(A)_n\to\rB(A)_{n-1}$
for the induced boundary map, and $\rE(A)_\norm$ and $\rB(A)_\norm$ for the
normalized complexes.

\subsection{The cyclic module of a cocommutative coalgebra}
If $C$ is a $k$-coalgebra, with counit $\epsilon:C\to k$
and coproduct $\Delta$, we have a
simplicial $k$-module $R(C)$, with $R_n(C)=C^{\otimes n+1}$, and face
and degeneracy operators given by
\begin{gather*}
\eps_i:R_n(C)\to R_{n-1}(C)\qquad (i=0,\dots,n)\\
\eps_i(c_0\otimes\dots\otimes c_n)=
\epsilon(c_i) c_0\otimes\dots\otimes c_{i-1}\otimes c_{i+1}\otimes\dots c_n\\
\Delta_i:R_n(C)\to R_{n+1}(C)\qquad (i=0,\dots,n)\\
\Delta_i(c_0\otimes\dots\otimes c_n)=c_0\otimes\dots
c_{i-1}\otimes c_i^{(0)}\otimes c_i^{(1)}\otimes\dots\otimes c_n
\end{gather*}
Here and elsewhere we use the (summationless) Sweedler notation
$\Delta(c)=c^{(0)}\otimes c^{(1)}$ of \cite{Sweedler}.

We remark that each $R_n(C)$ has a coalgebra structure,
and that the face maps $\eps_i$ are coalgebra homomorphisms.
%while the degeneracies $\Delta_i$ are bicomodule homomorphisms.
If in addition $C$ is cocommutative, then the degeneracies $\Delta_j$
are also coalgebra homomorphisms, and $R(C)$ is a simplicial coalgebra.
In fact $R_n(C)$ is the product of $n+1$ copies of $C$ in the category
of cocommutative coalgebras, and $R(C)$ is a particular case of the
usual product simplicial resolution of an object in a category with
finite products, which is a functor not only on the simplicial category
of finite ordinals and monotone maps, but also on the larger category
$\fin$ with the same objects, but where a homomorphism is just any set
theoretic map, not necessarily order preserving.
By \cite[6.4.5]{lod} we have:

\begin{lem} \label{map:tc}
For a cocommutative coalgebra $C$, the simplicial module $R(C)$
has the structure of a cyclic $k$-module, with cyclic operator
\begin{equation*}
\lambda(c_0\otimes\dots\otimes c_n)=
(-1)^nc_n\otimes c_0\otimes\dots\otimes c_{n-1}.
\end{equation*}
\end{lem}

\begin{ex}
Let $G$ be a group; write $k[G]$ for the group algebra. Note $k[G]$ is
a Hopf algebra, and in particular, a coalgebra, with coproduct
determined by %$\Delta:k[G]\to k[G]\otimes k[G]$,
$\Delta(g)=g\otimes g$. The cyclic module $R(k[G])$ thus defined is precisely
the cyclic module whose associated cyclic bicomplex was considered by
Karoubi in \cite[2.21]{kar}, where it is written $\tilde{C}_{**}(G)$.
\end{ex}

\subsection{The case of Hopf algebras}
Let $\Ha$ be a Hopf algebra with unit $\eta$, counit $\epsilon$ and
antipode $S$. We shall assume that $S^2=1$, which is the case for all
cocommutative Hopf algebras. We may view $R(\Ha)$ as
a simplicial left $\Ha$-module via the diagonal action:
\[
a\cdot (h_0\otimes\dots\otimes h_n)=a^{(0)}h_0\otimes\dots\otimes a^{(n)}h_n.
\]
Consider the maps (defined using summationless Sweedler notation):
\begin{gather}
\alpha:\rE_n(\Ha)\to R_n(\Ha),\nonumber\\
\alpha(h_0\otimes\dots\otimes h_n)=h_0^{(0)}\otimes
h_0^{(1)}h_1^{(0)}\otimes\dots\otimes
h_0^{(n)}h_1^{(n-1)}\!\cdots h_{n-1}^{(1)}h_n\label{alpha}\\
\beta:R_n(\Ha)\to \rE_n(\Ha),\nonumber\\ \beta(h_0\otimes\dots\otimes h_n)
=h_0^{(0)}\otimes(Sh_0^{(1)})h_1^{(0)}\otimes\dots\otimes
(Sh_{n-1}^{(1)})h_n\label{beta}
\end{gather}
A straightforward computation reveals:

\begin{lem}\label{lem:ab}
The maps \eqref{alpha} and \eqref{beta} are inverse isomorphisms of
simplicial $\Ha$-modules: $\rE_n(\Ha)\cong R_n(\Ha)$.
\end{lem}

\subsection{Cyclic complexes of cocommutative Hopf algebras}
From now on, we shall assume that $\Ha$ is a cocommutative Hopf
algebra. In this case the cyclic operator $\lambda:R(\Ha)\to R(\Ha)$
of \ref{map:tc} is a homomorphism of $\Ha$-modules. Thus $R(\Ha)$ is a cyclic
$\Ha$-module, and we can use the isomorphisms $\alpha$ and $\beta$
of Lemma \ref{lem:ab} to translate this
structure to the bar resolution $\rE(\Ha)$. We record this as a
corollary:

\begin{cor}\label{map:t}
When $\Ha$ is a cocommutative Hopf algebra, $\rE(\Ha)$ is a cyclic
$\Ha$-module, and $\rB(\Ha)=k\otimes_{\Ha}\rE(\Ha)$ is a cyclic
$k$-module. The  cyclic operator
$t:=\beta\lambda\alpha$ on $\rE(\Ha)$ is given by the formulas
$t(h)=h$, $t(h_0\otimes h_1)=-h_0h_1^{(0)}\otimes Sh_1^{(1)}$ and:
\begin{equation*}
t(h_0\otimes\dots\otimes h_n)=(-1)^nh_0 h_1^{(0)}\dots
h_n^{(0)}\otimes S(h_1^{(1)}\dots h_n^{(1)})\otimes h_1^{(2)}\otimes
\dots\otimes h_{n-1}^{(2)}.
\end{equation*}
\end{cor}
\smallskip

\begin{ex}
If $g_0,\dots, g_n\in\Ha$ are grouplike elements then
\[
t(g_0\otimes\dots\otimes g_n)=(-1)^ng_0\dots g_n\otimes
(g_1\dots g_n)^{-1}\otimes g_1\otimes\dots\otimes g_{n-1}.
\]
In particular, for $\Ha=k[G]$, $\rB(k[G])$ with the cyclic structure of
\ref{map:t} is the cyclic module associated to the cyclic set $\rB(G,1)$
of \cite[7.3.3]{lod}.
% and $\cM(k[G])$ is the corresponding normalized mixed complex.
\end{ex}

For the extra degeneracy of \cite[2.5.7]{lod},
\begin{equation}\label{map:s'}
s'=(-1)^{n+1}ts_n:\rE_n(\Ha)\to\rE_{n+1}(\Ha),
\end{equation}
it is immediate from \eqref{map:simpE} and \ref{map:t} that
$s'$ is signfree:
\[
s'(h_0\otimes\dots\otimes h_n)=h_0 h_1^{(0)}\dots h_n^{(0)}\otimes
S(h_1^{(1)}\dots h_n^{(1)})\otimes h_1^{(2)}\otimes\dots\otimes h_{n}^{(2)}.
\]
By \ref{map:t}, the Connes' operator $B'$ is the $\Ha$-module homomorphism:
\begin{equation}\label{map:B'unnorm}
B'= %(-1)^{n+1}(1-t)s'\sum_{i=0}^nt^i:\rE_n(\Ha)\to\rE_{n+1}(\Ha).
    (1-t)s'\sum_{i=0}^nt^i:\rE_n(\Ha)\to\rE_{n+1}(\Ha).
\end{equation}
(See \cite[p.\ 569]{LQ}.) %\edit{not \cite[2.5.10]{lod} !}
We write $B:\rB(\Ha)_n\to \rB(\Ha)_{n+1}$ for the induced $k$-module map.
\smallskip

\begin{defn}\label{def:M(H)}
We define the mixed $\Ha$-module complex
$\cM'=\cM'(\Ha)$, and the mixed $k$-module complex $\cM=\cM(\Ha)$
to be the normalized mixed complexes associated to the cyclic modules
$\rE_*(\Ha)$ and $\rB_*(\Ha)$ of \ref{map:t}:
\[
\cM'(\Ha):=(\rE_*(\Ha)_{\norm},\partial',B')
\]
%is a mixed complex of $\Ha$-modules, and
\[
\cM(\Ha):=k\otimes_\Ha\cM'(\Ha)=(\rB_*(\Ha)_\norm,\partial,B).
\]
%is a mixed complex of $k$-modules.
\end{defn}
\smallskip

\begin{rem}\label{rem:B'=B"} Consider the map
$s''=(-1)^ns_n:\rE_n(\Ha)\to \rE_{n+1}(\Ha)$, and set $B''=-ts''N$,
where as usual $N=\sum t^i$ is the norm map.
Then $B'=B''$ on $\rE(\Ha)_\norm$ because
the relations $s_0t_n=(-1)^nt_{n+1}^2s_n$ and $tN=N$ yield
for all $x\in\rE_n(\Ha)$:
\begin{align*}
(B''-B')(x)&=((-1)^{n+1}ts_nN+(-1)^n(1-t)ts_nN)(x)\\
             &=(-1)^n(-ts_nN(x)+ts_nN(x)-t^2s_nN(x))\\
             &=-s_0tN(x)=-s_0N(x)\equiv 0\text{\quad in $\rE(\Ha)_\norm$.}
\end{align*}
\end{rem}
\smallskip

\begin{lem}\label{map:B'}
The map $B':\rE(\Ha)_{\norm}\to \rE(\Ha)_{\norm}[1]$ induced
by \eqref{map:B'unnorm} is given by the explicit formula:
%\edit{Lemma \ref{map:B'} used only in \ref{B'(prim)}}
\begin{gather*}
B'(h_0\otimes\dots\otimes h_n)=\\
\sum_{i=0}^n(-1)^{ni}h_0h_1^{(0)}\!\cdots h_{n-i}^{(0)}\otimes
h_{n-i+1}^{(0)}\otimes\dots\otimes h_n^{(0)}\otimes
S(h_1^{(1)}\!\cdots h_n^{(1)})\otimes
h_1^{(2)}\otimes\dots\otimes h_{n-i}^{(2)}.\nonumber
\end{gather*}
\end{lem}

\begin{proof}
For convenience, let us write $Sh^{(1)}$ for $S(h_1^{(1)}\!\cdots h_n^{(1)})$.
It follows from \ref{map:t}, cocomutativity and induction on $i$ that
$t^i(h_0)=h_0$, and if $i\le n$ then:
\begin{gather*}
t^i(h_0\otimes \cdots \otimes h_n)= \\ (-1)^{ni}h_0h_1^{(0)}\dots
h_{n-i+1}^{(0)}\otimes h_{n-i+2}^{(0)}\otimes \dots \otimes h_n^{(0)}
\otimes Sh^{(1)} %S(h_1^{(1)}\dots h_n^{(1)})
\otimes h_1^{(2)}\otimes \dots \otimes h_{n-i}^{(2)}.
\end{gather*}
Let $s''=(-1)^ns_n$ be as in Remark \ref{rem:B'=B"} and set $m=n(i+1)+1$.
We have
\begin{gather*}
-ts''t^i(h_0\otimes \dots \otimes h_n)=
(-1)^{n+1}ts_nt^i(h_0\otimes \dots \otimes h_n)\\
=(-1)^{m}t(h_0h_1^{0}\dots h_{n-i+1}^{(0)}\otimes
h_{n-i+2}^{(0)}\otimes \dots \otimes h_n^{(0)}
\otimes Sh^{(1)} %S(h_1^{(1)}\dots h_n^{(1)})
\otimes h_1^{(2)}\otimes \dots \otimes h_{n-i}^{(2)}\otimes 1)\\
=(-1)^{ni}h_0h_1^{(0)}\dots h_{n-i}^{(0)}\otimes h_{n-i+1}^{(0)}
\otimes\dots\otimes h_{n}^{(0)}\otimes Sh^{(1)} %S(h_1^{(1)}\dots h_n^{(1)})
\otimes h_1^{(2)}\otimes\dots\otimes h_{n-i}^{(2)}.
\end{gather*}
Now sum up over $i$ to get $B''$ and use Remark \ref{rem:B'=B"}.
\quad\qed\end{proof}

\begin{cor}\label{B'(prim)}
Suppose that $x_1,\dots,x_n$ are primitive elements of $\Ha$,
and $h\in\Ha$.
Then $B'(h\otimes x_1\otimes\cdots\otimes x_n)=0~$ in $\rE(\Ha)_\norm$.
\end{cor}

\begin{proof} %\edit{used only in \ref{map:psi}}
%Since $B'$ is $\Ha$-linear, it suffices to prove the lemma for $h=1$.
When $x$ is primitive, $x^{(0)}\otimes x^{(1)} \otimes x^{(2)}$ is
$1\otimes1\otimes x + 2\otimes x \otimes1 + x\otimes1\otimes1$.
By Lemma \ref{map:B'}, $B'(h\otimes x_1\otimes\dots\otimes x_n)$
is a sum of terms of the form
\[
\pm h'\otimes x_{j+1}^{(0)}\otimes\dots\otimes x_n^{(0)}\otimes
S(x_1^{(1)}\!\cdots x_n^{(1)})\otimes x_1^{(2)}
\otimes\cdots\otimes x_{j}^{(2)}.
\]
By inspection, each such term is degenerate,
and vanishes in $\rE(\Ha)_\norm$.
\quad\qed\end{proof}

\subsection{Adic filtrations and completion}\label{subsec:adic}
As usual, we can use an adic topology on a Hopf algebra to define
complete Hopf algebras, and topological versions of the
above complexes.

First we recall some generalities about filtrations and
completions of $k$-modules, following \cite{Q-RHT}.  There is a
category of filtered $k$-modules and filtration-preserving maps; a
filtered module $V$ is a module equipped with a
decreasing filtration $V=\cF_0(V)\supseteq \cF_1(V)\supseteq\cdots$.
The completion of $V$ is $\hat{V}=\lim_nV/\cF_nV$; it is a filtered
module in the evident way.
If $W$ is another filtered $k$-module, then the tensor product
$V\otimes W$ is a filtered module with filtration
\begin{equation}\label{filotimes}
\cF_n(V\otimes W)=\sum_{p+q= n}\mathrm{image}(\cF_pV\otimes\cF_qW
\to  V\otimes W).
\end{equation}
We define $\hat{V}\hat{\otimes}\hat{W}$ to be
$\widehat{\hat{V}\otimes\hat{W}}$. Note that
\begin{equation}\label{filcompa}
\cF_n(V\otimes W)\supseteq \mathrm{image}\biggl(\cF_nV\otimes W
 + V\otimes\cF_nW \to V\otimes W\biggr)\supseteq\cF_{2n}(V\otimes W).
\end{equation}
Hence the topology defined by $\{\cF_n(V\otimes W)\}$ is the same as
that defined by $\{\ker(V\otimes W\to V/\cF_nV\otimes W/\cF_nW)\}$.
It follows that $\hat{V}\hat{\otimes}\hat{W}$ satisfies
\begin{equation}\label{hatotimes}
\hat{V}\hat{\otimes}\hat{W} =
\widehat{V\otimes W}.
\end{equation}
All this has an obvious extension to tensor products of finitely many
factors.

If $A$ is a filtered algebra (an algebra which is filtered as a $k$-module
by ideals), then $\hat{A}$ is an algebra. If $I$ is an ideal
then $\{ I^n\}$ is called the $I$-adic filtration on $A$.

Now suppose that $\Ha$ is a cocommutative Hopf algebra,
equipped with the $I$-adic filtration, where $I$ is both
a (2-sided) ideal and a (2-sided) coideal of $\Ha$, closed under the
antipode $S$ and satisfying $\epsilon(I)=0$.
The coideal condition on $I$ means that
\[
\Delta(I)\subset \Ha\otimes I+I\otimes \Ha.
\]
This implies that $\Delta:\Ha\to\Ha\otimes\Ha$ is filtration-preserving;
by \eqref{hatotimes} it induces a map
$\hat\Delta:\haha\to\haha\hat\otimes\haha$, making $\haha$ into a
{\it complete Hopf algebra} in the sense of \cite{Q-RHT}.

Consider the induced
filtrations \eqref{filotimes} in $\rE_n(\Ha)$ and $\rB_n(\Ha)$ ($n\ge0$).
It is clear, from the formulas \eqref{map:simpE} and
our assumption that $\epsilon(I)=0$, that the
simplicial structures of $\rE(\Ha)$ and $\rB(\Ha)$ are compatible with
the filtration; i.e., $\rE(\Ha)$ and $\rB(\Ha)$ are simplicial filtered
objects. Since $I$ is a coideal closed under $S$, the formula in
\ref{map:t} implies that $t$ preserves the filtration.
Thus $\rE(\Ha)$ and $\rB(\Ha)$ are
cyclic filtered modules. It follows that $\cM'(\Ha)$ and $\cM(\Ha)$
are filtered mixed complexes.
%We consider the completion of each of these complexes
% with respect to its filtration.

%Note that the hypothesis on $I$ implies that the
%comultiplication $\Delta$ maps $\cF_n\Ha\to \cF_n(\Ha\otimes\Ha)$;
%using this together with \eqref{filcompa} and \eqref{hatotimes}, we
%see that $\haha$ is a {\it complete topological Hopf algebra}; i.e., it
%is complete (with respect to the topology of the filtration
%$\cF_n\haha=\ker(\haha\to \Ha/I^n)$), and is equipped with a product
%$\haha\hotimes\haha\to \haha$, and a coproduct $\haha\to
%\haha\hotimes\haha$ as well as a unit, a counit and an antipode, all
%of which are continuous.
The identities \eqref{hatotimes} show that
the completed objects $\hat{\rE}(\Ha)$, $\hat{\rB}(\Ha)$, etc.\  depend
only on the topological Hopf algebra $\haha$, and can be regarded as
its topological bar resolution, bar complex etc.\ which are defined
similarly to their algebraic counterparts, but substituting $\hotimes$
for $\otimes$ everywhere. In this spirit, we shall write
$\rE^{\top}(\haha)$, $\rB^{\top}(\haha)$, etc., for $\hat{\rE}(\Ha)$,
$\hat{\rB}(\Ha)$, etc.

\medskip
\section{Comparison with the cyclic module of the algebra $\Ha$}
\label{sec2}

Let $\Ha$ be a cocommutative Hopf algebra.
In this section, we construct an injective cyclic module map
$\tau:\rB(\Ha)\to C(\Ha)$, from the cyclic bar complex
$\rB(\Ha)$ of \ref{map:t} %with $\rB_n(\Ha)=\Ha^{\otimes n}$,
to the canonical cyclic $k$-module $C(\Ha)$ %with $C_n(\Ha)=\Ha^{\otimes n+1}$
of the algebra underlying $\Ha$ (\cite[2.5.4]{lod}),
and a lift $c$ of $\tau$ to the negative cyclic complex $HN(\Ha)$ of $\Ha$.

%In the previous section, we have introduced a cyclic $k$-module
%structure on the bar complex $\rB(\Ha)$ of a Hopf algebra $\Ha$; we
%have $\rB_n(\Ha)=\Ha^{\otimes n}$. On the other hand, because $\Ha$ is
%an algebra, there is also a canonical cyclic $k$-module $C(\Ha)$ with
%$C_n(\Ha)=\Ha^{\otimes n+1}$ (\cite[2.5.4]{lod}).

In the group algebra case, these constructions are well-understood.
The cyclic module inclusion $\tau: \rB(k[G])\subset C(k[G])$ is given in
\cite[7.4]{lod}; see Example \ref{ex:tau} below. Goodwillie proved in
\cite{goo} that $\tau$ admits a natural lifting to a
chain map $c:\rB(k[G])\to HN(k[G])$ to the negative
cyclic complex, and that $c$ is unique up to natural homotopy. An
explicit formula for such a lifting was given by Ginot \cite{ginot},
in the normalized, mixed complex setting.

%In this section we show that both $\tau$ and its natural lifting $c$
% can be defined for general cocommutative Hopf algebras.

\subsection{A natural section of the projection
		$\mathrm HN(\cM'(\Ha))\to HH(\cM'(\Ha))$.}

Recall from Definition \ref{def:M(H)} that $\cM'=\cM'(\Ha)$ is a mixed
complex whose underlying chain complex is
$HH(\cM')\!=\!(\rE(\Ha)_\norm,\partial')$, and write $\pi'$ for the
projection from the negative cyclic complex $HN(\cM')$ to $HH(\cM')$.
%Write $HN(\cM'(\Ha))$ and $HH(\cM'(\Ha))\!=\!\rE(\Ha)_\norm$ for the
%negative cyclic and the Hochschild complex of the mixed complex
%$\cM'(\Ha)$, and $\pi':HN(\cM'(\Ha))\to HH(\cM'(\Ha))$ for the projection.
Following the method of Ginot \cite{ginot},
we shall define a natural $\Ha$-linear chain homomorphism
$\Upsilon':HH(\cM')\to HN(\cM')$ such that $\pi'\,\Upsilon'=1$.

We shall use a technical lemma about maps between chain complexes of
modules over a $k$-algebra $A$.  Assume given all of the following:
\begin{enumerate}
\item[(i)] A homomorphism of chain $A$-modules $f:C\to D$,
with $C_n=0$ for $n< n_0$.
\item[(ii)] A decomposition $C_n\cong A\otimes V_n$ for each $n$,
where $V_n$ is a $k$-module.
\item[(iii)] A $k$-linear chain contraction $s$ for $D$.
% i.e., maps $s_n:D_n\to D_{n+1}$ such that $1_D=sd+ds$.
\end{enumerate}

\begin{lem}\label{lem:kappa}
Given (i)--(iii), there is an $A$-linear chain contraction $\kappa^f$
of $f$, defined by $\kappa^f_{n_0}(av)=asf(v)$ and the inductive formula:
\[%begin{gather*}
%\kappa^f_{n_0}:C_{n_0}\to D_{n_0+1},\quad \kappa^f_{n_0}(av)=asf(v),\\
 \kappa^f_n:C_n=A\otimes V_n\to D_{n+1},\quad
	\kappa^f_n(av)=a\, s\left(f-\kappa^f_{n-1}d\right)(v).
\]%end{gather*}
% That is, $f=\kappa^f d+d\kappa^f$.
\end{lem}

\begin{proof}
We have to verify the formula $f(av)=\kappa^fd(av)+d\kappa^f(av)$
for $a\in A$ and $v\in V_n$. When $n=n_0$, this is easy
as $d(av)=0$ and $f(v)=dsf(v)$:
%\item{Step $1$:} The case $m=n_0$. For $v \in V_{n_0}$,
\[
d\kappa^f_{n_0}(av)=adsf(v)= af(v) = f(av).
\]
Inductively, suppose that the formula holds for $n\!-\!1$. %and that $v\in V_{n+1}$.
Since $dv\in C_{n-1}$, we have
\[
d(f-\kappa^f d)(v) = f(dv)-d\kappa^f(dv)=(\kappa^f d)(dv) = 0.
%d\kappa^f(dv) = (f-\kappa^f d)(dv) = f(dv) = df(v).
\]
%Hence $d(f-\kappa^f d) =0$ on $V_{n}$.
Using this, and the definition of $\kappa^f_{n}$, we compute:
\[
(d\kappa^f_{n})(v) = ds(f-\kappa^f d)(v) = (1-sd)(f-\kappa^f d)(v) =
(f-\kappa^f d)(v).
\]
Since $\kappa^f d(av)=\kappa^f(a\,dv)=a\kappa^f(dv)$ by construction,
\[
 d\kappa^f(av)+\kappa^f d(av)=ad\kappa^f(v)+a\kappa^f(dv)=af(v) = f(av).
%\qedhere
\quad\qed\]
\end{proof}\goodbreak

\begin{lem}
There is a sequence of $\Ha$-linear maps
$\Upsilon'^n:\rE(\Ha)\to \rE(\Ha)[2n]$, starting with $\Upsilon'^0=1$,
such that $B'(\Upsilon'^n\,\partial'-\partial'\Upsilon'^n)=0$.

They induce maps on the normalized complexes
$\Upsilon'^n:HH(\cM')\to HH(\cM')[2n]$.
\end{lem}

\begin{proof}
Inductively, we suppose we have constructed $\Upsilon'^{n}$ satisfying
$B'[\Upsilon'\,^n,\partial']=0$.
Now any chain map from $C=\rE(\Ha)$ to $\rE(\Ha)[2n+1]$
%$HH(\cM') \to HH(\cM')[2n+1]$
must land in the good truncation $D=\rE(\Ha)[2n+1]_{\ge0}$,
%$D=HH(\cM')[2n+1]_{\ge0}$,
and the $k$-linear
chain contraction $-s$ of \eqref{map:s} is also a contraction of $D$.
We claim that the $\Ha$-linear map $f=-B'\Upsilon'\,^{n}: C\to D$
is a chain map. Since the differential on $D$ is $-\partial'$,
the claim follows from:
\[
f(\partial')  - (-\partial')f
= -B'\Upsilon'^n\partial' - \partial'\,B'\Upsilon'^n
= B'\left( \partial'\Upsilon'^n - \Upsilon'^n\partial' \right)
= 0.
\]
We define $\Upsilon'\,^{n+1}$ to be the $\Ha$-linear chain
contraction of $f=-B'\Upsilon'\,^{n}$ given by the formulas in
Lemma \ref{lem:kappa}. That is,
\begin{equation}\label{map:upsiloni}
%\Upsilon'^0=1,\qquad
\Upsilon'\,^{n+1} := \kappa^f=\kappa^{-B'\Upsilon'\,^{n}}\qquad (n\ge0).
\end{equation}
The chain contraction condition $[\Upsilon'\,^{n+1},\partial']=f$ for
\eqref{map:upsiloni} implies that the inductive hypothesis
$B'[\Upsilon'\,^{n+1},\partial']=B'f=0$ holds.

Finally, note that the normalized mixed complex $HH(M')$ is a quotient of
$\rE(\Ha)$, and its terms have the form $HH(M')_n=\Ha\otimes W_n$ for a
quotient module $W_n$ of $V_n$. By naturality of $\kappa^f$ in $f$, the
above construction also goes through with $\rE(\Ha)$ replaced by $M'(\Ha)$,
and the maps $\Upsilon'$ on $\rE(\Ha)$ and $HH(M')$ are compatible.
\quad\qed\end{proof}

We define maps $\Upsilon':HH(\rE(\Ha))\to HN(\rE(\Ha))$ and
$\Upsilon':HH(\cM')\to HN(\cM')$ by
\begin{equation}\label{map:upsilont}
\Upsilon'=\sum\nolimits_{n=0}^\infty \Upsilon'\,^n. %:HH(\cM')\to HN(\cM').
\end{equation}
That is, %Identifying $HN(M')_i=\prod_{n=0}^\infty M'_{i+2n}$,
$\Upsilon'(x)$ is $(\cdots,\Upsilon'^{n}(x),\cdots,\Upsilon'^{1}(x),x)$.
\smallskip

\begin{lem}\label{lem:Upsilon}
The maps $\Upsilon'$ in \eqref{map:upsilont} are morphisms of
chain $\Ha$-modules, and $\pi'\,\Upsilon'=1$. Here $\pi'$ is the appropriate
canonical projection, either $HN(\rE(\Ha))\to HH(\rE(\Ha))$ or
$HN(\cM')\to HH(\cM')$.
\end{lem}

\begin{proof}
It is clear that $\Upsilon'$ is $\Ha$-linear and that
$\pi'\Upsilon'=\Upsilon'^0=1$. To see that it is a chain map,
we observe that the $n$th coordinate of
$(B'+\partial')\Upsilon'-\Upsilon'\partial'$ is
$B'\Upsilon'\,^{n-1} +\partial'\Upsilon'\,^{n} - \Upsilon'\,^{n}\,\partial'$.
This is zero by the chain contraction condition for \eqref{map:upsiloni}.
\quad\qed\end{proof}

\begin{ex}\label{ex:Upsilon(1)}
Write $[1]^n$ for the element $1\otimes\cdots\otimes1$ of $k^{\otimes n}$.
By induction, we may check that
%\edit{same as $\Upsilon_0$ in \cite[2.4(ii)]{ginot}}
$\Upsilon'^n(1)=(-1)^n (2n)!/n!\,[1]^{2n+1}$.
Thus $\Upsilon'(1)=(0,\dots,0,1)$ in $HN(M')$.
\end{ex} %see \cite[(2.1.7.3)]{lod}

Recall from Definition \ref{def:M(H)} that $M=k\otimes_{\Ha}M'$,
and that $\rB(A)=k\otimes_A\rE(A)$.

\begin{cor}\label{cor:Upsilon}
There are morphisms of chain $k$-modules,
$\Upsilon\!:HH(\rB(\Ha))\to HN(\rB(\Ha))$ and
$\Upsilon\!:HH(\cM)\to HN(\cM)$, defined by
\[
\Upsilon=\sum\nolimits_{n=0}^\infty 1_k\otimes_\Ha\Upsilon'^n,
\]
and $\pi\Upsilon=1$. Here $\pi$ is
the appropriate projection $\pi\!:HN\smap{} HH$.
\end{cor}

\subsection{The lift $HH(\rB(\Ha))\map{c} HN(\Ha)$}

%Let $\Ha$ be a Hopf algebra, and
Recall that $C(\Ha)$ denotes the canonical cyclic complex of the algebra
underlying $\Ha$ (\cite[2.5.4]{lod}). We set $\tau_0=\eta:k\to\Ha$.

%we define a $k$-linear map
%\begin{gather}
%\tau:\rB(\Ha)\to C(\Ha)\label{map:tau}\\
%\tau(h_1\otimes\dots\otimes h_n)=S(h_1^{(0)}\dots h_n^{(0)})\otimes
%h_1^{(1)}\otimes\dots\otimes h_{n-1}^{(1)}\nonumber
%\end{gather}

\begin{lem}\label{map:tau}
Let $\Ha$ be a cocommutative Hopf algebra. Then the $k$-linear map
\begin{gather*}
\tau:\rB(\Ha)\to C(\Ha)\\
\tau(h_1\otimes\dots\otimes h_n)=S(h_1^{(0)}\dots h_n^{(0)})\otimes
h_1^{(1)}\otimes\dots\otimes h_{n}^{(1)},\quad n>0,\nonumber
\end{gather*}
is an injective homomorphism of cyclic $k$-modules.
It induces an injection of the associated mixed complexes,
$M(\Ha)\hookrightarrow C(\Ha)_\norm$.
\end{lem}

\begin{proof} One has to check that $\tau$ commutes with the face,
degeneracy and cyclic operators; these are all straightforward, short
calculations. The fact that the maps are injective follows from the
antipode identity $(Sh^{(0)})h^{(1)}=\eta\epsilon(h)$.
\quad\qed\end{proof}

\begin{ex}\label{ex:tau}
If $g_1,\dots,g_n\in\Ha$ are grouplike, then
\[
\tau(g_1\otimes\dots\otimes g_n)=(g_1\dots g_n)^{-1}\otimes
g_1\otimes\dots\otimes g_n.
\]
Thus for $\Ha=k[G]$, the $\tau$ of \ref{map:tau} is the map
$k[\rB(G,1)]\hookrightarrow HH(k[G])$ of \cite[7.4.5]{lod}.
\end{ex}

\smallskip\goodbreak
%\begin{defn}
We define $c: \rB(\Ha)\to HN(\Ha)$ to be the natural chain map
\begin{equation}\label{map:c}
c:\rB(\Ha) %\twoheadrightarrow \rB(\Ha)_\norm
\map{\Upsilon} HN(\rB(\Ha)) \map{\tau} HN(\Ha).
\end{equation}
We will also write $c$ for the normalized version
$HH(M)\to HN(\Ha)_\norm$ of this map.
%\end{defn}

\begin{thm}\label{thm:gwlift}
The following diagram commutes
\[
\xymatrix{& HN(\Ha)\ar[d]^\pi\\
\rB(\Ha)\ar[r]_\tau\ar[ur]^{c}& HH(\Ha).}
\]
\end{thm}

\begin{proof}
By \eqref{map:c}, Lemma \ref{map:tau} and Corollary \ref{cor:Upsilon},
$\pi\,c = \pi\tau\Upsilon=\tau\pi\Upsilon=\tau$.
\quad\qed\end{proof}

\begin{ex}\label{Gmap:c}
Goodwillie proved in \cite[II.3.2]{goo} that, up to chain homotopy,
there is a unique chain map $\rB(k[G])\to HN(k[G])$ lifting $\tau$,
natural in the group $G$.  Ginot \cite{ginot} has given explicit
formulas for one such map; it follows that Ginot's map is naturally
chain homotopic to the map $c$ constructed in \eqref{map:c} for $\Ha=k[G]$.
\end{ex}

\subsection{Passage to completion}\label{subsec:passage}

If $A$ is a filtered algebra, the induced filtration \eqref{filotimes}
%and $\hat{A}$ its completion with respect to a two-sided ideal $I$.
%The filtration of $A$ by the powers of $I$ induces a filtration
on the canonical cyclic module $C(A)$ makes it a cyclic filtered module.
Passing to completion we obtain a cyclic module $C^{\top}(\hat{A})$ with
$C^{\top}_n(\hat{A})=\hat{A}^{\hotimes n+1}$. In the spirit of subsection
\ref{subsec:adic},  we write $HH^{\top}(\hat{A})$, $HN^{\top}(\hat{A})$, etc.\
for the Hochschild and cyclic complexes etc.\ of the %normalized
mixed complex associated to $C^{\top}(\hat{A})$.

In particular this applies if $A=\Ha$ is a cocommutative Hopf algebra,
equipped with an $I$-adic filtration, where $I$ is an ideal and coideal
of $\Ha$ with $\epsilon(I)=0$, closed under the antipode $S$. Write
$\haha$ for the associated complete Hopf algebra.

It is clear from the formula in Lemma \ref{map:tau} that $\tau$ is a
morphism of cyclic filtered modules. Hence it induces continuous maps
$\hat{\tau}$ between the corresponding complexes for $HH$, $HN$, etc.

\begin{prop}
The map $c$ of \eqref{map:c} induces a continuous map $\hat{c}$ which
fits into a commutative diagram
\[
\xymatrix{& HN^{\top}(\haha)\ar[d]^\pi\\
\rB^{\top}(\haha)\ar[r]_{\hat{\tau}}\ar[ur]^{\hat{c}}& HH^{\top}(\haha).}
\]
\end{prop}

\begin{proof}
It suffices to show that $\Upsilon$ (and hence $c$) is a filtered
morphism.  We observed in \ref{subsec:adic} that $\rE(\Ha)$ is a
cyclic filtered module. Similarly, it is clear that $s$, $s'$ and $B'$
are filtered morphisms from their definitions in \eqref{map:s},
\eqref{map:s'} and \eqref{map:B'unnorm}. The recursion formulas in
Lemma \ref{lem:kappa} and \eqref{map:upsiloni} show that each
$\Upsilon'^{n}$ is filtered, whence so are $\Upsilon'$ and $\Upsilon$,
as required.
\quad\qed\end{proof}

\section{The case of universal enveloping algebras of Lie algebras}
\label{sec3}

Let $\frakg$ be a Lie algebra over a commutative ring $k$. Then the
enveloping algebra $U\frakg$ is a cocommutative Hopf algebra, so
the constructions of the previous sections apply to $U\g$. In
particular a natural map $c:\rB(U\g)\to HN(U\g)$ was constructed in
\eqref{map:c}. In this section, we show that the Loday-Quillen map
\[
\wedge\g\map{\theta} C^\lambda(U\g)[-1] \map{B} HN(U\g)
\]
factors through $c$ up to chain homotopy. (See Theorem \ref{thm:theta=c}.)

\subsection{Chevalley-Eilenberg complex}
The Chevalley-Eilenberg resolution of $k$ as a $U\g$-module has the
form $(U\g\otimes\wedge \frakg,d')$, and is given in \cite[7.7]{chubu}.
Tensoring it over $U\frakg$ with $k$, we obtain a complex
$(\wedge\frakg,d)$.  Kassel showed in \cite[8.1]{kas} that the
anti-symmetrization map
\begin{gather}
e:\wedge^n\frakg\to (U\frakg)^{\otimes n}\label{map:e}\\
e(x_1\land\dots\land x_n)=\sum_{\sigma\in\fS_n}\sg(\sigma)
x_{\sigma(1)}\otimes\dots\otimes x_{\sigma(n)}\nonumber
\end{gather}
induces chain maps $e:\wedge\frakg\to \rB(U\frakg)$ and
$1\otimes e:U\g\otimes\wedge\frakg\to \rE(U\frakg)$, because
$ed=\partial e$ and $(1\otimes e) d'=\partial'(1\otimes e)$.
Moreover, $e$ and $1\otimes e$ are quasi-isomorphisms; see \cite[8.2]{kas}.

%\begin{lem}\label{exa:Bprim=0}
%Let $x_1,\dots,x_n\in\frakg$ and $h\in U\frakg$. Then
%\[
%B'(h\otimes x_1\otimes\dots\otimes x_n)=0
%\text{ in }\rE_{n+1}(U\frakg)_{\norm}.
%\]
%\end{lem}

%\begin{proof}
%This is Corollary \ref{B'(prim)}.
%In fact, as the $x_i$ are primitive, each of these is degenerate,
%and so it vanishes in $\rE(U\g)_\norm$.
%\end{proof}

\begin{lem}\label{map:psi} %was {upsilon}
The map $\psi': U\g\otimes\wedge\g \to HN(\cM'(U\g))$ defined by
\[ \psi'(x)=(\dots,0,0,0,1\otimes e(x)) \]
is a morphism of chain $U\g$-modules, and the map
$\psi\!:\!\wedge\g\to HN(\cM(U\g))$ defined by
\[ \psi(x)=(\dots,0,0,0,e(x)) \]
is a morphism of chain $k$-modules.
\end{lem}

\begin{proof}
Consider $U\g\otimes\wedge\g$ and $\wedge\g$ as mixed
complexes with trivial Connes' operator. By Corollary \ref{B'(prim)},
$B'(1\otimes e)=0$ in $M'(U\g)$. Thus
both $1\otimes e$ and $e$ induce morphisms of mixed complexes
$U\g\otimes\wedge\g\to \cM'(U\g)$ and $\wedge\g\to\cM'(U\g)$.
\quad\qed\end{proof}

\begin{lem}\label{lem:ce=taupsi} %was {lem:upsitopic}
The diagrams
\[\xymatrix{
U\g\otimes\wedge\g\ar[d]_{1\otimes e}\ar[r]^-{\psi'}
 & HN(\cM'(U\g)) \\
 \rE(U\g)\ar[ur]_{\Upsilon'}
} \qquad \xymatrix{
\wedge\g\ar[d]_{e}\ar[r]^-{\psi}
 & HN(\cM(U\g))\ar[d]^\tau \\
 \rB(U\g)\ar[ur]^{\Upsilon}\ar[r]_-c & HN(U\g)_\norm
} \]
commute up to natural $U\g$-linear (resp., natural $k$-linear) chain homotopy.
%\edit{used in T.\ref{thm:theta=c}}
\end{lem}

\begin{proof} By \ref{cor:Upsilon}, \ref{map:psi} and \eqref{map:c},
it suffices to consider the left diagram.

Consider the mixed subcomplex $N\subset\cM'(U\frakg)$
given by $N_0=0$, $N_1=\ker\partial'$
and $N_n=\rE_n(U)_\norm$. Because $\psi'(1)=\Upsilon'(1)$ by Example
\ref{ex:Upsilon(1)}, the difference
$f=\Upsilon'(1\otimes e)-\psi'$ factors through $HN(N)$. Put
$\phi^n=(-1)^n (sB')^ns:N\to N[2n+1]$.
One checks that $\phi:=\sum_{i=0}^\infty\phi^n$ is a natural,
$k$-linear contracting homotopy for $HN(N)$. Now apply Lemma \ref{lem:kappa}.
\quad\qed\end{proof}

%\begin{cor}\label{cor:Upsi-upsi}
%The diagram
%\[
%\xymatrix{&HN(\cM(U\g))\\
%\wedge\g\ar[ur]^{\upsilon}\ar[r]_{e}&\rB(U\g)\ar[u]_{\Upsilon}}
%\]
%commutes up to natural $k$-linear homotopy.
%\end{cor}

\goodbreak%\subsection{Compatibility}
There are simple formulas for $\tau$ and $\tau\,\psi$
in the normalized complexes.

\begin{lem}\label{lem:taux}
Let $x_1,\dots,x_n\in\g$. Then in $C(U\g)_\norm$ we have:
\[
\tau(x_1\otimes\dots\otimes x_n)=1\otimes x_1\otimes\dots\otimes x_n
\]
\end{lem}

\begin{proof} %\edit{used only in \ref{cor:taupsi}}
Let $\nabla^{(n)}:U\g^{\otimes n}\to U\g$ be the multiplication map
%$[~]: C\to C_{\norm}$ the projection,
and $\sigma\in\fS_{2n}$
the (bad) riffle shuffle $\sigma(2i-1)=i$, $\sigma(2i)=n+i$.
By definition (see \ref{map:tau}),
\begin{equation}\label{eq:compotau}
\tau=(S\otimes 1^{\otimes n})\circ (\nabla^{(n)}\otimes
1^{\otimes n})\circ\sigma\circ\Delta^{\otimes n}
\end{equation}
in $C(U\g)$. Since the $x_i$ are primitive,
\[
\Delta^{\otimes n}(x_1\otimes\dots\otimes x_n)=(x_1\otimes 1+
1\otimes x_1)\otimes\dots\otimes(x_n\otimes 1+1\otimes x_n)
\]
Expanding this product gives a sum in which
\[
x=1\otimes x_1\otimes 1\otimes x_2\otimes\dots\otimes 1\otimes x_n
 \]
is the only term not mapped to a degenerate element of $C(U\g)$
under the composition \eqref{eq:compotau}. Thus in $C(U\g)_\norm$ we have:
\begin{gather*}
\tau(x_1\otimes\dots\otimes x_n)=(S\otimes 1^{\otimes n})
(\nabla^{(n)}\otimes 1^{\otimes n})\sigma(x)\\
  =(S\otimes 1^{\otimes n})(\nabla^{(n)}\otimes 1^{\otimes n})
(1\otimes\dots\otimes 1\otimes x_1\otimes\dots\otimes x_n)\\
  =(S\otimes 1^{\otimes n})(1\otimes x_1\otimes\dots\otimes x_n)=
1\otimes x_1\otimes\dots\otimes x_n.~ \rlap{\quad\qed} %\qedhere
\end{gather*}
\end{proof}

\begin{cor}\label{cor:taupsi} %was {lem:ce=taupsi}
%The following diagram is naturally homotopy commutative.
%\[
%\xymatrix{\wedge \g\ar[d]_e\ar[r]^-\psi & HN(\cM(U\g))\ar[d]^\tau\\
% \rB(U\g)\ar[r]^-c\ar[ur]_\Upsilon & HN(U\g)_\norm }
%\]
%Moreover,
We have:~~
$ %\begin{equation}\label{map:ce}
\tau\psi(x_1\land\dots\land x_n)=
(\dots,0,0,0,1\otimes e(x_1\land\dots\land x_n)).
$ %\end{equation}
\end{cor}

\begin{proof} %\edit{used only in Thm.\ref{thm:theta=c}}
%By Lemma \ref{lem:ce=taupsi}, $\tau\psi=\tau\Upsilon e=ce$.
Combine Lemmas \ref{map:psi} and \ref{lem:taux}.
\quad\qed\end{proof}

\subsection{The Loday-Quillen map}

We can now show that $\tau\,\psi$ factors through the %normalized
Connes' complex $C^\lambda(U\g)=\coker(1-t:U^{\otimes\ast}\to HH(U\g))$.
We have a homomorphism $\theta$ which lifts
the Loday-Quillen map of \cite[10.2.3.1, 11.3.12]{lod} to $U\g$:
\begin{gather*}
\theta:\wedge^{n+1}\g\to C^\lambda_{n}(U\g)\\
\theta(x_0\land x_1\land\dots\land x_n)=x_0\otimes e(x_1\land\dots\land x_n).
\end{gather*}
Because we are working modulo the image of $1-t$, $\theta$ is well defined.
The following result is implicit in the proof of \cite[10.2.4]{lod}
for $r=1$, $A=U\g$.

\begin{lem}
$\theta$ is a chain homomorphism $\wedge\g \to C^\lambda(U\g)[-1]$.
\end{lem}

\begin{proof}
To show that $b\theta=-\theta d$, we fix a
monomial $x_1\land\dots\land x_n$ and compute that
$b\theta(x_0\land\dots\land x_n)=\sum_{\sigma\in\fS_n}
b(x_0\otimes x_{\sigma(1)}\otimes\dots\otimes x_{\sigma(n)}) =A+B$,
where $A$ equals:
\begin{gather*}
\sum_{\sigma\in\fS_n}\sg(\sigma)\biggl(
  x_0x_{\sigma(1)}\otimes x_{\sigma(2)}\otimes\dots\otimes x_{\sigma(n)}
+(-1)^nx_{\sigma(n)}x_0\otimes
  x_{\sigma(1)}\otimes\dots\otimes x_{\sigma(n-1)}\biggr)
\\
=\sum_{\sigma\in\fS_n}\sg(\sigma)\biggl(
  x_0x_{\sigma(1)}\otimes x_{\sigma(2)}\otimes\dots\otimes x_{\sigma(n)}
 - x_{\sigma(1)}x_0\otimes
	x_{\sigma(2)}\otimes\dots\otimes x_{\sigma(n)}\biggr) \\
= \sum_{i=1}^n\sum_{\sigma\in\fS\atop\sigma(i)=i}
   (-1)^{i-1}\sg(\sigma)[x_0,x_i]\otimes x_{\sigma(1)}\otimes\cdots\otimes
x_{\sigma(i-1)}\otimes x_{\sigma(i+1)}\otimes\dots\otimes x_{\sigma(n)}\\
= \sum_{i=1}^n(-1)^{i-1}[x_0,x_i]\otimes e(x_1\land\dots\land x_{i-1}
\land x_{i+1}\land\dots\land x_n),
\end{gather*}
and $B$ equals
\begin{gather*}
\sum_{\sigma\in\fS_n}\sum_{i=1}^{n-1}(-1)^i\sg(\sigma)
   x_0\otimes x_{\sigma(1)}\otimes\dots\otimes x_{\sigma(i)}x_{\sigma(i+1)}
	\otimes\dots\otimes x_{\sigma(n)}
\\
%\sum_{\sigma\in\fS_n}
= x_0\otimes e(d(x_1\land\dots\land x_n)).
\end{gather*}
Similarly, $\theta\,d(x_0\land\dots\land x_n)$ is the sum of $-A$ and
\[
\sum_{0<i<j\le n}\!(-1)^{i+j+1}
x_0\otimes e(
[x_i,x_j]\wedge x_1\wedge\cdots\wedge x_n)
=-x_0\otimes e(d(x_1\land\dots\land x_n))
\]
which equals $-B$.  Therefore
$b\theta(x_0\land\dots\land x_n)=-\theta d(x_0\land\dots\land x_n)$.
% is the sum of this term and
%$\sum_{\sigma\in\fS_n}x_0\otimes e(d(x_1\land\dots\land x_n))$,
%which equals $\theta(d(x_0\land\dots\land x_n))$.
\quad\qed\end{proof}

\noindent{\it Warning:}
the sign convention used for $d$ in \cite[10.1.3.3]{lod}
differs by $-1$ from the usual convention, used here and in
\cite[7.7]{chubu} and \cite{kas}.
\smallskip

\smallskip
It is well known and easy to see that, because Connes' operator $B$
vanishes on the image of $1-t$, it induces a chain map
for every algebra $A$:
\begin{gather*}
B:C^\lambda(A)[-1]\to HN(A)\\
B[x]=(\dots,0,0,0,Bx).
\end{gather*}

Let $\wedge^+\g$ denote the positive degree part of $\wedge\g$.
%\begin{lem}\label{lem:taupsi=Btheta}
%$\tau\,\psi=B\,\theta$ as chain maps $\wedge^+\g\to HN(U\g)_\norm$,
%where $\wedge^+\g$ is the positive degree part of $\wedge\g$.
%\end{lem}

\begin{thm}\label{thm:theta=c}%\label{lem:taupsi=Btheta}
We have $\tau\,\psi=B\,\theta$ as chain maps $\wedge^+\g\to HN(U\g)_\norm$.

Hence the following diagram commutes up to natural chain homotopy.
\[ \xymatrix{
\wedge^+\g\ar[dr]^{c\circ e}\ar[d]_\theta\\
C^\lambda(U\g)[-1]\ar[r]^B & HN(U\g)_\norm.
%& HN(U\g)\\ \wedge\g\ar[ur]^{c\circ e}\ar[dr]_\theta &\\ &
%C^\lambda(U\g)[-1]\ar[uu]_B
} \]
\end{thm}

\noindent
Theorem \ref{thm:theta=c} fails for the degree~0 part $k$ of $\wedge\g$.
Indeed, $\theta(1)=0$ for degree reasons,
while from Example \ref{ex:Upsilon(1)} we see that
$c\,e(1)=\tau\,\psi(1)=(\dots,0,1)$ is nonzero.

\begin{proof}
%\edit{used only in Ex.\ref{ex:rho}\\which is used in Th.\ref{thm:main}}
By Lemma \ref{lem:ce=taupsi}, it suffices to check that $\tau\,\psi=B\,\theta$.
By Corollary \ref{cor:taupsi} and \eqref{map:e}, we have
\begin{equation}\label{eq:taupsilon}
\tau\psi(x_1\land\dots\land x_n)= (\dots,0,0,
	\sum_{\sigma\in\fS_n}\sg(\sigma)1\otimes
	x_{\sigma(1)}\otimes\dots\otimes x_{\sigma(n)}).
\end{equation}
Note the expression above contains no products in $U\g$; neither do the
formulas for $\theta(x_0\land x_1\land\dots\land x_n)$ and the Connes'
operator in $C(U\g)$. Thus we may assume that $\g$ is abelian, and that
$U\g=S\g$ is a (commutative) symmetric algebra.
We may interpret $\theta(x_1\land\dots\land x_n)$ as the
shuffle product $x_1\star B(x_2)\star\dots\star B(x_n)$
(see \cite[4.2.6]{lod}), and
the nonzero coordinate of \eqref{eq:taupsilon} as
the shuffle product %\cite[4.2.6]{lod}
\begin{gather*}
(1\otimes x_1)\star\dots \star (1\otimes x_n)
=B(x_1)\star\dots\star B(x_n) \qquad\qquad \\
~=~ B\bigl(x_1\star B(x_2)\star\dots\star B(x_n)\bigr) \quad
\text{ ~~by \cite[3.1]{LQ} or \cite[4.3.5]{lod}} \\
=B(\theta(x_1\land\dots\land x_n)).  \qquad\qed\qquad\qquad %\qedhere
\end{gather*}
%Since $x_1\star B(x_2)\star\dots\star B(x_n)=\theta(x_1\land\dots\land x_n)$,
%this equals $B(\theta(x_1\land\dots\land x_n))$.
\end{proof}

%\begin{thm}\label{thm:theta=c}
%The following diagram commutes up to natural chain homotopy.
%\[ \xymatrix{
%\wedge\g\ar[r]^{c\circ e}\ar[d]_\theta & HN(U\g)\\
%C^\lambda(U\g)[-1]\ar[ur]_B
%} \]
%\end{thm}
%\begin{proof} Immediate from Lemmas \ref{lem:ce=taupsi}
%and \ref{lem:taupsi=Btheta}.
%\quad\qed\end{proof}

%____________________________________________________
\section{Nilpotent Lie algebras and nilpotent groups}\label{sec4}
\numberwithin{equation}{section}

In this and the remaining sections we shall fix the ground ring $k=\Q$.
Let $\g$ be a nilpotent Lie algebra;
consider the completion $\hat{U}\g$ of its enveloping algebra with
respect to the augmentation ideal, and set
\[
G=\exp{\g}\subset \hat{U}\g.
\]
This is a nilpotent group, and $\Ha=\Q[G]$ is a Hopf algebra.
The inclusion $G\subset\hat{U}\g$ induces a homomorphism
$\Ha=\Q[G]\hookrightarrow\hat{U}\g$, and $\haha\cong\hat{U}\g$
by \cite[A.3]{Q-RHT}.
%By subsection \ref{subsec:passage}, we have chain maps from
%both $HN(\Q[G])$ and $HN(U\g)$ to $HN^{\top}(\haha)$.

On the other hand, Suslin and Wodzicki showed
in \cite[5.10]{sw} that there is a natural quasi-isomorphism
$sw:\rB(\Q[G])\to \wedge\g$. Putting these maps together with those
considered in the previous sections, we get a diagram
\begin{equation}\label{diag:nil1}
\xymatrix{
\rB(\Q[G])\ar[d]_{sw}\ar[rr]^{c}\ar[dr] && HN(\Q[G])\ar[dr]&\\
\wedge\g\ar[d]_e &\rB^{\top}(\haha)\ar[rr]^{\hat{c}}&&HN^{\top}(\haha).\\
           B(U\g)\ar[ur]\ar[rr]_{c}&&HN(U\g)\ar[ur]&}
%\xymatrix{\rB(\Q[G]\ar[d]_{sw}\ar[r]^{c}&HN(\Q[G])\ar[dr]&\\
%\wedge\g\ar[d]_e&&HN^{\top}(\haha). \\
%\rB(U\g)\ar[r]^{c}&HN(U\g)\ar[ur]&}
\end{equation}

\begin{prop}\label{prop:nil1}
Diagram \eqref{diag:nil1} commutes up to natural chain homotopy.
\end{prop}

\begin{proof}
% We have maps
%\[
%\xymatrix{\rB(\Q[G])\ar[d]_{sw}\ar[rr]^{c}\ar[dr]&&HN(\Q[G])\ar[dr]&\\
%\wedge\g\ar[d]_e&\rB^{\top}(\hat{U}\g)\ar[rr]^{\hat{c}}&&HN^{\top}(\hat{U}\g)\\
%           B(U\g)\ar[ur]\ar[rr]_{c}&&HN(U\g)\ar[ur]&}
%\]
%The outer edge ofthis diagram is \eqref{diag:nil1};
The two parallelograms commute by naturality.
The triangle on the left of \eqref{diag:nil1} commutes up to
natural homotopy by Lemma \ref{lem:nil2} below.
\quad\qed\end{proof}

%\bigskip
%Consider the maps $\Q[G]\to \hat{U}\g$ and $U\g\to\hat{U}\g$.
%We have a diagram
%\begin{equation}\label{diag:nil2}
%\xymatrix{\rB(\Q[G])\ar[d]_{sw}\ar[r]&\rB^{\top}(\haha)\\
%           \wedge\g\ar[r]^e&\rB(U\g)\ar[u]_{\iota}}
%\end{equation}

\begin{lem}\label{lem:nil2}
The following diagram %Diagram \eqref{diag:nil2}
commutes up to natural chain homotopy.
\begin{equation*}%\label{diag:nil2}
\xymatrix{\rB(\Q[G])\ar[d]_{sw}\ar[r]&\rB^{\top}(\haha)\\
           \wedge\g\ar[r]^e&\rB(U\g)\ar[u]_{\iota}}
\end{equation*}
\end{lem}

\begin{proof}
By construction (see \cite[5.10]{sw}), the map $sw$ is induced by a
map $\rE(\Q[G])\to \haha\hotimes\wedge\g$. Let $\iota$ be the upward
vertical map; $\iota e$ is induced by $1\otimes\iota
e:\haha\hotimes\wedge\g\to \rE^{\top}(\haha)$. Thus it suffices
to show that the composite $\rE(\Q[G])\to \rE^{\top}(\haha)$ is
naturally homotopic to the map induced by the homomorphism
$\Q[G]\to\haha$. By definition, these maps agree on $\rE_0(\Q[G])$;
thus their difference goes to the subcomplex
$\ker(\hat{\epsilon}:\rE^{\top}(\haha)\to \Q)$ which is
contractible, with contracting homotopy induced by
\eqref{map:s}. Hence we may apply Lemma \ref{lem:kappa}; this finishes
the proof.
\quad\qed\end{proof}

\begin{ex}\label{ex:rho}
Let $\g=J_{\mathrm{Lie}}$ be the Lie algebra associated to an ideal $J$
in a $\Q$-algebra $A$ ($\g$ is $J$ with the commutator bracket); there is a
canonical algebra map $U\g\to A$ sending $\g$ onto $J$.
If $J$ is a nilpotent ideal then $\g$ is a nilpotent Lie algebra and
the induced algebra map $\hat{U}\g\to A$ restricts to an
isomorphism $G=\exp(\g)\iso (1+J)^{\times}$, as is proven in \cite[5.2]{sw}.
%\edit{used in \ref{lem:jc-ch}, \eqref{map:rht}\\ proof of Th.\ref{thm:main}}

Let $C(A,J)$ denote the kernel of $C(A)\to C(A/J)$, and let
$C^{\lambda}(U\g,\cI)$ denote the kernel of
$C^{\lambda}(U\g)\to C^{\lambda}(k)$.
The composite of $\theta$ with the map induced by $U\g\to A$ factors
through $C^\lambda(A,J)$, giving rise to a commutative diagram
\[ \xymatrix{
\wedge\g\ar[d]_{\rho}\ar[r]^-{\theta}& C^{\lambda}(U\g,\cI)[-1]\ar[d]\ar[dl]\\
C^\lambda(A,J)[-1]\ar@{^{(}->}[r]&C^\lambda(A)[-1].
} \]
The composite $\Q[G]\to\hat{U}\g\to A$ sends the augmentation ideal
$\cI_G$ of $\Q[G]$ to $J$.
Consider the resulting map
\[
j:HN(\Q[G],\cI_G)\to HN(\hat{U}\g,\hat{\cI})\to HN(A,J).
\]
Putting together Theorem \ref{thm:theta=c} with Proposition \ref{prop:nil1},
we get a naturally homotopy commutative diagram (with $G=(1+J)^\times$):
\[\xymatrix{
%\rB(\Q[G],\cI_G)\ar[r]^-{c}\ar[d]_{sw}& HN(\Q[G],\cI_G)\ar[dr]^j\\
%\wedge^+\g\ar[dr]^-\theta\ar[r]^-{c\circ e} & HN(U\g,\cI)_\norm\ar[r]
%   & HN(\hat{U}\g,\hat{\cI})_\norm\ar[r]	& HN(A,J)_\norm \\
%   & C^{\lambda}(U\g,\cI)[-1]\ar[rr]\ar[u]_B &&C^{\lambda}(A,J)[-1].\ar[u]_B
%
\rB(\Q[G],\cI_G)\ar[r]^-{c}\ar[d]_{sw}& HN(\Q[G],\cI_G)\ar[dr]^j\\
\wedge^+\g\ar[dr]^-\theta\ar[r]^-{c\circ e} & HN(U\g,\cI)_\norm\ar[r]
	& HN(A,J)_\norm \\
   & C^{\lambda}(U\g,\cI)[-1]\ar[r]\ar[u]_B &C^{\lambda}(A,J)[-1].\ar[u]_B
} \]
\end{ex}

\goodbreak
\section{The relative Chern character of a nilpotent ideal}\label{sec5}
\numberwithin{equation}{subsection}

In this section we establish Theorem \ref{thm:main}, promised in
\eqref{intro:agree}, that the two definitions \eqref{intro:ch} and
\eqref{intro:ch'} of the Chern character $K_*(A,I)\to HN_*(A,I)$
agree for a nilpotent ideal $I$ in a unital $\Q$-algebra $A$.
The actual proof is quite short, and most of this section is devoted
to the construction of the maps \eqref{intro:ch} and \eqref{intro:ch'}.

For this it is appropriate to regard a non-negative chain complex $C$ as a
simplicial abelian group via Dold-Kan, identifying the homology of the complex
with the homotopy groups $\pi_*(C)$ by abuse of notation.
If $X$ is a simplicial set, we write $\Z[X]$ for its singular complex,
so that $H_*(X;\Z)$ is $\pi_*\Z[X]$,
and the Hurewicz map is induced by the simplicial map $h:X\to\Z[X]$.
%Note that if $X_0$ is a point,
%then the good truncation $Z_{\ge 1}[X]$
%is equivalent to the reduced singular complex:
%\[ \Z_{\ge1}[X]\map{\sim}\Z[X]/\Z. \]

\subsection{The absolute Chern character}

Let $A$ be a unital $\Q$-algebra,
%$\GL_n(A)$ the general linear group, $\GL(A)=\bigcup_n\GL_n(A)$,
and $\rB\GL(A)$ the classifying space of $GL(A)$.
Now the plus construction $\rB\GL(A)\to \rB\GL(A)^+$ is a homology
isomorphism, and $K_n(A)=\pi_n\rB\GL(A)^+$ for $n\ge1$. In particular,
the singular chain complex map $\Z[\rB\GL(A)]\to\Z[\rB\GL(A)^+]$ is
a quasi-isomorphism. As described in \cite[11.4.1]{lod},
 the absolute Chern character $ch_n:K_n(A)\to HN_n(A)$
(of Goodwillie, Jones et al.) is the composite $ch=ch^{-}_A\circ h$
of the Hurewicz map $h:\rB GL(A)^+\to \Z[\rB GL(A)^+]\simeq\Z[\rB\GL(A)]$,
the identification $\Z[\rB G]=\rB(\Z[G])$ with the bar complex,
and the chain complex map $ch^-_A$, which is defined as the
stabilization (for $GL_n\subset GL_{n+1}$) of the composites:
\begin{equation}\label{map:chA}
\xymatrix @C=1.3pc{
%\Z[\rB \GL(A)^+]\ar[d] &
\rB(\Z[\GL_n(A)])\ar[r]^{c}\ar[d] %\ar[l]_{\cong}
& HN(\Z[\GL_n(A)])\ar[r]^(.6){}\ar[d] & HN(M_n(A))\ar[r]^-{\tr}
& HN(A)_\norm.\\
%\Q[\rB \GL_n(A)^+] &
\rB(\Q[\GL_n(A)])\ar[r]^c %\ar[l]_{\cong}
& HN(\Q[\GL_n(A)])\ar[ur] & }
\end{equation}
Here $c$ is the natural map defined in \eqref{map:c} for $k=\Z$ and $\Q$;
the middle maps in the diagram are induced by the fusion maps
$\Z[\GL_n(A)]\subset\Q[\GL_n(A)]\to M_n(A)$, and $\tr$ is the trace map.
The maps $HN(\Q[\GL_n(A)])\to HN(A)_\norm$ are independent of $n$ by
\cite[8.4.5]{lod}, even though the fusion maps are not.

%When $A$ is a $\Q$-algebra, $HN(A)$ is a complex of $\Q$-modules,
%and $ch^-$ factors through the localized sequence:
%\[
%\xymatrix{\Q_{\ge1}[\rB GL(A)^+] &
%\rB_{\ge1}(\Z[\GL(A)])\ar[l]_{\cong}\ar[r]^c &
%HN_{\ge 1}(\Z[\GL(A)])\ar[r]^(.6){\tr} & HN_{\ge 1}(A)}.
%\]

\begin{ex}
If $A$ is commutative and connected, the composition of the rank map
$K_0(A)\to\Z$ sending $[A^r]$ to $r$ with the map $H_0(ch^-_A)$ yields
the Chern character $ch^-_0:K_0(A)\to HN_0(A)$ of \cite[8.3]{lod}.
Composing this with the maps $HN(A)\to HC(A)[2n]$ yields the map
$ch_{0,n}: K_0(A)\to HC_{2n}(A)$ of \cite[8.3.4]{lod}.
From Example \ref{ex:Upsilon(1)} above, with $A=k$, we see that
$ch([k])=c(1)$, and $ch_{0,n}([k])=(-1)^n(2n)!/n!$ in $HC_{2n}(k)\cong k$,
in accordance with \cite[8.3.7]{lod}.
\end{ex}

\subsection{Volodin models for the relative Chern character
	of nilpotent ideals}

In order to define the relative version $ch_*$ of the absolute
Chern character, %\eqref{intro:ch},
we need to recall a chunk of notation about Volodin models.
For expositional simplicity, we
shall assume that $I$ is a nilpotent ideal in a unital $\Q$-algebra $A$.

\begin{defn}
Let $I$ a nilpotent ideal in a $\Q$-algebra $A$, and $\sigma$ a
partial order of $\{1,\dots,n\}$.  We let $\cT_n^\sigma(A,I)$
be the nilpotent subalgebra of $M_n(A)$ defined by:
\[
\cT_n^\sigma(A,I):=\{a\in M_n(A): a_{ij}\in I \text{ if } i\nless_\sigma j\}.
\]
We write $\frakt_n^\sigma(A,I)$ for the associated Lie algebra, and
$T_n^\sigma(A,I)$ for the group
\[ T_n^\sigma(A,I) = \exp\frakt_n^\sigma(A,I)= 1+\cT_n^\sigma(A,I)
\subset GL_n(A).
\]
\end{defn}

%Note $\cT_n^\sigma(A,I)\subset M_n(A)$ is a nonunital subring.  If
%$I=A$, it is the full matrix ring; $\cT_n^\sigma(A,A)=M_n(A)$. At the
%other extreme, we write $\cT_n^{\sigma}(A):=\cT_n^{\sigma}(A,0)$; if
%$\sigma$ is the usual order, this is the usual ring $\cT_n(A)$ of
%upper triangular matrices.  Note $\cT_n^\sigma(A,I)$ is the pullback
%of $\cT_n^{\sigma}(A/I)\subset M_n(A/I)$ under $M_n(A)\to M_n(A/I)$.
%If $\sigma$ is a total order, we may identify it with a permutation
%$\sigma\in\fS_n$; the latter group acts on $M_n(A)$, and
%$\cT_n^{\sigma}(A,I)$ is the conjugate of $\cT_n(A,I)$ under the action
%of $\sigma$.  Given any ideal $I$ of $A$ and any partial order
%$\sigma$, we can also consider $\cT_n^{\sigma}(A,I)$ as a Lie ring
%with the commutator bracket; this gives the triangular Lie ring:
%\[
%\frakt_n^\sigma(A,I)=\cT_n^\sigma(A,I)_{\mathrm{Lie}}
%\]
%If $A$ happens to be an algebra over some ground ring $k$, then
%$\frakt_n^\sigma(A,I)$ will be a Lie algebra over $k$. The triangular
%group associated to $\sigma$ and $I$ is
%\[
%T_n^\sigma(A,I)=1+\cT_n^\sigma(A,I)\subset
%(\Z\cdot 1+\cT_n^\sigma(A,I))^\times\subset\GL_n(A)
%\]

%\subsection{Volodin models and the relative Chern character}
The {\it Volodin space} $X(A)\subset BGL(A)$ is defined to be the union
of the spaces $X_n(A)=\bigcup_\sigma %{\sigma\in\fS_n}
BT_n^\sigma(A,0)$; see \cite[11.2.13]{lod}. The {\it relative Volodin space}
$X(A,I)$ is defined to be the union of the spaces
$\bigcup_\sigma %{\sigma\in\fS_n}
BT_n^\sigma(A,I)$; see \cite[11.3.3]{lod}.
%As usual, $X(A,I)$ denotes the union of the $X_n(A,I)$.

The morphism $ch^-_A:\Q[B\GL(A)]\to HN(A)_\norm$
of \eqref{map:chA} sends the subcomplex $\Q[X(A,I)]$
to $HN(A,I)_\norm$, which is the kernel of $HN(A)_\norm\to HN(A/I)_\norm$;
see \cite[11.4.6]{lod}. The chain map $ch^-$ is
defined to be the restriction of $ch^-_A$:
\begin{equation}\label{map:chm}
ch^-:\Q[X(A,I)]\to HN(A,I)_\norm.
\end{equation}

It will be useful to have a more detailed description of the restriction
of \eqref{map:chm} to $\Q[BT_n^\sigma(A,I)]$. Recall that if $\Lambda$ is a
non-unital subalgebra of $R$ then $\Q+\Lambda$ is a unital subalgebra of $R$;
we write $C(\Lambda)$ for the cyclic submodule $C(\Q+\Lambda,\Lambda)$
of $C(\Q+\Lambda)$ and hence of $C(R)$, following \cite[2.2.16]{lod}.
When $\Lambda=\cT_n^\sigma(A,I)$ and $R=M_n(A)$, we obtain the cyclic
submodule $C(\cT_n^\sigma(A,I))$ of $C(M_n(A))$.

The trace map $C(M_n(A))\to C(A)$ is a morphism of cyclic modules,
sending $C(\cT_n^\sigma(A,I))$ to the submodule $C(A,I)$.
%The composite of
%$HN(\cT_n^\sigma(A,I))\subset HN(M_n(A))$ followed by the trace sends
%$HN(\cT_n^\sigma(A,I))$ inside $HN(A,I)$. We abuse notation and write
%$\tr:HN(\cT_n^\sigma(A,I))\to HN(A,I)$ for the induced map.
On the other hand, by Example \ref{ex:rho}, we also have a map
\[
C(\Q[G],\cI_{G}) \map{j} C(G), \quad\text{for~} G=\cT_n^\sigma(A,I).
%C(\Q[T^\sigma_n(A,I)],\cI_{T^\sigma_n(A,I)}) \map{j} C(\cT_n^\sigma(A,I)).
\]
From the definition of $ch^-_A$ in \eqref{map:chA}
and the naturality of $c$,
we obtain the promised description of $ch^-$, which we record.

\begin{lem}\label{lem:jc-ch}
Set $G=T_n^\sigma(A,I)$. The restriction of $ch^-$ to $\rB(\Q[G],\cI_G)$
is the composition
\[
\rB(\Q[G],\cI_G) \map{c} HN(\Q[G],\cI_G) \map{j} HN(\cT_n^\sigma(A,I))
\map{\tr} HN(A,I)_\norm.
\]
\end{lem}

\subsection{The relative Chern character for rational nilpotent ideals}

When $I$ is a nilpotent ideal in an algebra $A$, we define
$K(A,I)$ to be the homotopy fiber of
$B\GL(A)^+\to B\GL(A/I)^+$; %our assumption that $I$ is nilpotent implies that
$K(A,I)$ is a connected space whose homotopy groups %$\pi_nK(A,I)$
are the relative $K$-groups $K_n(A,I)$ for all $n$.
We now cite Theorem 6.1 of \cite{OW} for nilpotent $I$; the proof in \cite{OW}
is reproduced on page 361 of \cite{lod}.

\begin{thm}\label{OW6.1}
If $I$ is a nilpotent ideal in $A$, there are homotopy fibrations
\begin{gather*} X(A,I) \to B\GL(A) \to B\GL(A/I)^+, \\
X(A) \to X(A,I) \to K(A,I).
\end{gather*}
Moreover, $X(A,I)^+ \map{\sim} K(A,I)$ and
$\Z[X(A,I)]\map{\sim}\Z[K(A,I)]$ are
homotopy equivalences (i.e., $X(A,I)\to K(A,I)$ is a homology isomorphism).
\end{thm}

\begin{defn}\label{map:chrel} (See \cite[11.4.7]{lod})
The {\it relative Chern character} for the ideal $I$ of a $\Q$-algebra
$A$ is the %\edit{$\Q$ not $\Z$ 7/20}
composite of the Hurewicz map, the inverse of the homotopy equivalence
of Theorem \ref{OW6.1} and the map $ch^-$ of \eqref{map:chm}:
\begin{equation*}
\xymatrix{ch:K(A,I)\ar[r]^h&\Q[K(A,I)]&
\Q[X(A,I)]\ar[l]_{\sim}^{\ref{OW6.1}}\ar[r]^{ch^-}&HN(A,I)_\norm.}
\end{equation*}
\end{defn}
%This is the relative character described in \cite[11.4.5-7]{lod}.

\subsection{The rational homotopy theory character for nilpotent ideals}

For a nilpotent ideal $I$, consider the chain subcomplex of the
Chevalley-Eilenberg complex $\wedge\mathfrak{gl}(A)$,
\[
x(A,I)=\sum_{n,\sigma} \wedge\frakt_n^\sigma(A,I).
\]
Because $sw$ is natural in $G$, the family of maps
$\rB(\Q[T_n^\sigma(A,I)])\map{sw} \wedge\frakt_n^\sigma(A,I)$
induces a morphism of complexes
\[
sw_X:\Q[X(A,I)]\to x(A,I).
\]
On the other hand, for each $n$ and $\sigma$, the composite of the map
$B\rho:\wedge\frakt_n^\sigma(A,I)\to HN(\cT_n^\sigma(A,I))$
%$B\rho:\wedge^+\frakt_n^\sigma(A,I)\to HN_{\ge 1}(\cT_n^\sigma(A,I))$
%\edit{OK?  7/18} %$\wedge\frakt_n^\sigma$, $HN(A,I)$?}
of Example \ref{ex:rho} with the inclusion and the trace, i.e., with
\[
HN(\cT_n^\sigma(A,I))\subset HN(M_n(A)) \map{tr} HN(A),
%HN_{\ge1}(\cT_n^\sigma(A,I))\subset HN_{\ge1}(M_n(A)) \map{tr} HN_{\ge1}(A),
\]
sends $\wedge\frakt_n^\sigma(A,I)$ into the subcomplex $HN(A,I)$.
%sends $\wedge^+\frakt_n^\sigma(A,I)$ into the subcomplex $HN_{\ge1}(A,I)$.
%\edit{OK without $\ge1$!}
All of these maps are natural in $n$ and $\sigma$;
by abuse of notation, we write $\tr(B\,\rho)$ for the resulting map:
\[
\tr(B\,\rho):x(A,I)\to HN(A,I).
\]

\begin{defn}\label{map:rht}
The map $ch^-_{\rht}: \Q[X(A,I)]\to x(A,I)\to HN(A,I)_\norm$
is defined to be $\tr(B\,\rho)\circ sw_X$,
followed by $HN(A,I)\to HN(A,I)_\norm$.
The {\it rational homotopy theory character} of \cite[11.3.1]{lod},
cited in \eqref{intro:ch'}, is the composite
$ch':K(A,I) \to HN(A,I)_\norm$ defined by:
\begin{equation*}
K(A,I) \map{h}
\Q(K(A,I)) \ {\buildrel\simeq\over\leftarrow} \ \Q[X(A,I)]
%\Q_{\ge1}(K(A,I)) \ {\buildrel\simeq\over\leftarrow} \ \Q_{\ge1}[X(A,I)]
% \lmap{sw_X} x_{\ge1}(A,I) \lmap{\tr(B \rho)}
\lmap{ch^-_{\rht}} HN(A,I)_\norm. %HN_{\ge 1}(A,I).
\end{equation*}%\edit{deleted $\ge1$!}
\end{defn}

\begin{rem}
We will not need the unnormalized version of $ch'$. %\edit{7/18}
By construction, $ch'$ is the map
$K(A,I)\to C^\lambda(A,I)[-1]$ of \cite[A.13]{chwinf},
followed by Connes' operator $B$.
\end{rem}

\subsection{Main theorem}

Let $I$ be a nilpotent ideal in a $\Q$-algebra $A$. The relative Chern
character $ch$ of Definition \ref{map:chrel} induces the relative Chern
character $ch_*$ of \eqref{intro:ch} on homotopy groups, and the
rational homotopy character $ch'$ of Definition \ref{map:rht} induces
the character $ch'_*$ of \eqref{intro:ch'} on homotopy groups.
Therefore the equality $ch_*=ch'_*$ of \eqref{intro:agree} is an immediate
consequence of our main theorem.

\begin{thm}\label{thm:main}
The maps $ch^-$ and  $ch^-_{\rht}$ are naturally chain homotopic.
Hence the maps $ch$ and $ch'$ are
homotopic for each $A$ and $I$. %\edit{normalized version?}
\end{thm}

\begin{proof}
%It suffices to show that the map $ch^-_{\rht}$
%is naturally chain homotopic to the relative map $ch^-$ of \eqref{map:chm}.
We first consider the restriction of $ch^-$ and $ch^-_\rht$ to
$\rB(\Q\,[T_n^\sigma(A,I)])$ for some fixed
$n$ and $\sigma$. By Lemma \ref{lem:jc-ch}, the restriction of $ch^-$ to
$\rB(\Q\,[T_n^\sigma(A,I)])$ is the map $\tr(jc)$; by Example \ref{ex:rho},
there is a natural chain homotopy from $jc$ to $B\,\rho\circ sw$.
Since $\tr(B\,\rho)\circ sw$ is the restriction of $ch^-_{\rht}$ to
$\rB(\Q\,[T_n^\sigma(A,I)])$, the chain homotopies glue together
by naturality to give the desired chain homotopy
%$\Q_{\ge1}[X(A,I)]\to HN_{\ge 1}(A,I)[1]$
from $ch^-_{\rht}$ to $ch^-$.
\quad\qed\end{proof}

\subsection{Naturality}
In order to formulate a naturality result for the homotopy between
$ch$ and $ch'$, it is necessary to give definitions for the maps $ch$ and
$ch'$ which are natural in $A$ and $I$.
Contemplation of Definitions \ref{map:chrel} and \ref{map:rht} shows
that we need to find a natural inverse for the backwards quasi-isomorphism
of Theorem \ref{OW6.1}. One standard way is to fix a small category of
pairs $(A,I)$ (say all pairs with a fixed cardinality bound on $A$) and
consider the global model structure on the category of covariant functors
from this category to Simplicial Sets; this will yield naturality with
respect to all morphisms in the small category.

Let $K(A,I)'$ be the cofibrant replacement of $K(A,I)$, and factor the
backwards map as \linebreak
$\Q\,[K(A,I)]\ {\buildrel\simeq\over\twoheadleftarrow}\
C {\buildrel\simeq\over\leftarrowtail} \Q[X(A,I)]$.  Then $h$ lifts to
a map $h':K(A,I)'\to C$ and, since $HN(A,I)$ is fibrant,
$ch^-$ and $ch^-_\rht$ both lift to maps $C\to HN(A,I)$. We can then
define $ch$ to be the composite $K(A,I)' \map{h'} C \map{ch^-} HN(A,I)$;
$ch'$ is defined similarly using $ch^-_\rht$. By Theorem \ref{thm:main},
there is a homotopy between the two maps $C\to HN(A,I)$, and hence between
the two maps $K(A,I)'\to HN(A,I)$. Since this is a homotopy of functors
from the small category of pairs $(A,I)$ to Simplicial Sets, it provides
a natural homotopy between $ch$ and $ch'$ as maps $K(A,I)'\to HN(A,I)$.

\section*{Acknowledgments}
We would like to thank Gregory Ginot and Jean-Louis Cathelineau for
bringing the $ch_*=ch'_*$ problem to our attention. We are also grateful
to Christian Haesemeyer for several discussions about the relation between
this paper and our joint paper \cite{chwinf}.

\end{document}